\documentclass[11pt,twoside,a4paper]{article}
\usepackage{hyperref}
\usepackage{algorithm}
\usepackage{algorithmic}
\usepackage[a4paper,left=2cm,right=2cm,top=2cm,bottom=2cm]{geometry}

\usepackage{setspace}

\usepackage{comment}
\usepackage[utf8]{inputenc}
\usepackage{stmaryrd} 
\usepackage{amsmath}
\usepackage{amsfonts}
\usepackage{bbold}
\usepackage{amssymb}
\usepackage{array}
\usepackage[T1]{fontenc}
\usepackage{graphicx}
\usepackage{caption}
\usepackage{subcaption}
\usepackage[english]{babel}
\usepackage[overload]{empheq}
\usepackage[normalem]{ulem}
\usepackage{comment}
\usepackage{color}
\usepackage{stmaryrd}
\usepackage{tikz,pgf}
\usepackage{lmodern}
\usepackage[justification=centering]{caption}
\usepackage{amsthm}
\usepackage{import}
\usepackage{bbm}

\usepackage{lineno}
\usepackage{csquotes}
\usepackage{authblk}
\usepackage{rotating}
\usepackage{setspace}

\newtheorem{The}{Theorem}[section]
\newtheorem{Pro}[The]{Proposition}
\newtheorem{Hypo}[The]{Hypothesis}
\newtheorem{Lem}[The]{Lemma}

\title{Long time behavior of a degenerate stochastic system modeling  the response of a population face to environmental impacts}
\author{Pierre Collet\thanks{CPHT, Ecole Polytechnique, CNRS, Institut polytechnique de Paris, route de
Saclay, 91128 Palaiseau Cedex-France; E-mail:
pierre.collet@cpht.polytechnique.fr}\,, Claire Ecoti\`ere\thanks{CREST, ENSAE, CNRS, Institut polytechnique de Paris,  5 Avenue Henry Le Chatelier 91120 Palaiseau, France; E-mail: claire.ecotiere@ensae.fr}\,,
Sylvie M\'el\'eard\thanks{CMAP, Ecole Polytechnique, CNRS, Institut polytechnique de Paris, Inria, route de
Saclay, 91128 Palaiseau Cedex-France and Institut Universitaire de France; E-mail: 
{sylvie.meleard@polytechnique.edu}}\,.}
\date{September 2023}

\begin{document}

\maketitle

\begin{abstract}
We study the asymptotics of a two-dimensional stochastic differential system with a degenerate diffusion matrix. This system describes the dynamics of a population where individuals  contribute to the degradation of their environment through two different behaviors. We exploit the almost one-dimensional  form of the dynamical system to compute explicitly the Freidlin-Wentzell action functional. 
That allows to give conditions under which the small noise regime of the  invariant measure is concentrated around the equilibrium of the dynamical system having the smallest diffusion coefficient. 
\end{abstract}

Keywords : Stochastic differential system - Large deviations - Small noise regime - Invariant measure. 

\section{Introduction}
Reducing the effects of global change requires the adoption by most of the human population of consistent proenvironmental behaviors. But we, as individuals, tend to act in reponse to alarming events, and relax when things seem to get better. In \cite{ecotiere} Ecotiere et al. have explored and numerically quantified how the tendency to behave inconsistently  can be countered by social interactions and social pressure. In that aim, they  developed a simple two-dimensional mathematical model describing the coupled dynamics of the perceived environmental state over time $(e_t, t\ge 0)$, and the repartition of a fixed size population between two behaviors $A$ and $B$ more or less active  from an environmental perception. Expressing behavior $A$ makes the agent reduce its environmental impact (compared to $B$). The macroscopic frequency dynamics of behavior $A$ in the population  is modeled by the deterministic function $(x_t, t\ge 0)$ taking values in $[0,1]$. The model is as follows:
\begin{equation}
    \begin{aligned}
			\frac{dx_t}{dt} &=p(x_t,e_t)=\kappa x_t(1-x_t)(\lambda_A(x_t )-\lambda_B(x_t))+\tau_A(e_t)(1-x_t)-\tau_B(e_t)x_t \\
			\frac{de_t}{dt} &=h(x_t,e_t)=\ell e_t(l_Ax_t+l_B(1-x_t)-e_t).\\
		\end{aligned}
		\label{sys_dyn}
\end{equation}
The interpretation of the parameters has been carefully  commented in \cite{ecotiere} and we refer to that paper for more details. The deterministic system \eqref{sys_dyn} is the large population approximation of a stochastic system $((X^N_t, E^N_t), t\ge 0)$, where the scaling parameter is the population size $N$  which is assumed to tend to infinity. In this stochastic setting, the population dynamics is described by  a pure jump Markov process whose jump rates model the individual behavior changes and depend on the perceived environmental state. The coupled stochastic system is a piecewise deterministic Markov process where the deterministic environmental component satisfies
$$E^N_t = E^N_0 +\int_0^t h(X^N_s, E^N_s)ds.$$

We are interested in the long time behavior of the stochastic system $((X^N_t, E^N_t), t\ge 0)$, and also  how it is related, for large $N$ to  the long time behavior of the deterministic system \eqref{sys_dyn}. To that goal,  we  have considered the approximation-diffusion of the process (see Ethier-Kurz p.354 for a precise definition) which consists in a two-dimensional process, solution of the stochastic differential system

\begin{equation}
	\left\{
	\begin{aligned}
		dX^N_t&= P(X^N_t,E^N_t)dt+\frac{\sigma(X^N_t,E^N_t)}{\sqrt{2N}}dB_t\\ 
		dE^N_t&= h(X^N_t,E^N_t)dt, \\
	\end{aligned}
	\right.  
	\label{Approx diff}
\end{equation} 
where $B$ is a Brownian motion and
$$\sigma^2(x,e)=\kappa x(1-x)(\lambda_A(x )+\lambda_B(x))+\tau_A(e)(1-x)+\tau_B(e)x.$$

System \eqref{Approx diff} depends on so many parameters that we needed to simplify it by successive parametrizations.
We have  proven in Appendix 1  that, under some assumptions on the parameters and using different parametrizations and changes of variables, the system can be simplified in the following form:
\begin{equation}
	\left\{
	\begin{aligned}
		dX^{\varepsilon}_t&= \lambda_1X^{\varepsilon}_t(1-(X^{\varepsilon}_t)^2)dt+\varepsilon\cos(\theta)\left(\sigma_0+\sigma_1X^{\varepsilon}_t\right)dB_t \\ 
		dY^{\varepsilon}_t&= \left(-\lambda_2 Y^{\varepsilon}_t+ \lambda_3 X^{\varepsilon}_t(1-(X^{\varepsilon}_t)^2) \right)dt+\varepsilon\sin(\theta)\left(\sigma_0+\sigma_1X^{\varepsilon}_t\right)dB_t, \\
	\end{aligned}
	\right.  
	\label{sys_intro}
\end{equation}  
where $B$ is a standard Brownian motion and $\varepsilon$ is a small parameter ($\varepsilon$ of the order of $1/\sqrt{N}$). 
Note hat this system is degenerate for two reasons: first, the two coordinates are driven by the same Brownian motion,  and second, the diffusion coefficient in front of the Brownian motion  cancels on the line $x=-\sigma_0/\sigma_1$; all this leading to a non-invertible diffusion matrix. 

We are interested in the long time behaviour of the process $(X^{\varepsilon},Y^{\varepsilon})$ when $\varepsilon$ is small. The limits $\varepsilon$ tending to $0$ and $t$ tending to infinifty a priori don't  commute and one can ask about 
i) the long time behavior of the process $(X^{\varepsilon},Y^{\varepsilon})$ for fixed $\varepsilon$,
ii) its limit when $\varepsilon$ tends to $0$,
iii) the relation with the long time behavior of the deterministic system 
\begin{equation}
	\left\{
	\begin{aligned}
		dx_t&= b_{1}(x_{t}) = \lambda_1x_t(1-(x_t)^2)dt \\ 
		dy_t&= b_{2}(x_{t}, y_{t}) = \left(-\lambda_2 y_t+ \lambda_3 x_t(1-(x_t)^2) \right)dt, \\
	\end{aligned}
	\right.  
	\label{sys_main}
\end{equation}  
obtained from \eqref{sys_intro} when  $\varepsilon=0$.

\medskip
In Section \ref{section_main}, we state the different results answering to these questions. 
It's standard to prove the convergence of the  stochastic differential system \eqref{sys_intro} when $\varepsilon$ tends to 0, towards the deterministic system \eqref{sys_main}. The latter admits three equilibrium points: two stable points framing a saddle point.
For a fixed $\varepsilon$, we  show the existence, uniqueness and exponential convergence,  of an invariant probability measure $\pi^\varepsilon$, absolutely continuous with respect to the Lebesgue measure.
We also show that the sequence $(\pi^\varepsilon)_\varepsilon$ converges when $\varepsilon$ tends to 0 to  a linear combination of the Dirac measures on the stable equilibrium points of the deterministic system. In particular, we show that  the sign of $\sigma_1$ determines on which of the two stable equilibria, the limit invariant measure $\pi$ is concentrated.


\medskip
The study of invariant measures appears in various fields, ranging from potential wells in physics \cite{hairer2009hot,hairer2009slow} to neural networks (\cite{hopfner2017strongly,holbach2020positive}). 
For non-degenerate stochastic differential systems, we refer to the book by \cite{freidlin1998random} which details the steps to show that the invariant measure of the stochastic differential system is concentrated outside the unstable equilibrium points of the perturbed deterministic system such that col nodes and sources.
This method is based on control theory (see \cite{stroock1972support}) and the calculation of the energy necessary to move from a compact containing one equilibrium point to another one.
The energy translated by the action functional involves the inverse of the diffusion matrix of the stochastic differential system.
In our degenerate case, it is not possible to apply these theorems. We will therefore seek to extend these results to our case study. For this, we will use the particular form of our system which has a main component on its two dimensions.



In Section \ref{section_exist_inv_mes} and \ref{section_limit_supp_inv_mes}, we respectively prove the existence and uniqueness of the system's limit when $t$ tends to infinity, then we study the limit of the invariant measures $\pi^\varepsilon$ when $\varepsilon$ tends to $0$.

\section{Main results}\label{section_main}

Let us state the main results of the paper.
Let $\varepsilon_0$ positive and fixed through the paper, and we consider $\varepsilon\in ]0,\varepsilon_0[$. Recall that the process $(X^{\varepsilon},Y^{\varepsilon})$ is solution of \eqref{sys_intro} 
with $B$  a standard (uni-dimensional) Brownian motion, issued from   $(x_0,y_0)\in \mathbb{R}^2$. The diffusion coefficient for both coordinates  is proportional to
\begin{equation}
\sigma_0+\sigma_1 x,
	\label{coeff_diffusion}
\end{equation} 
where $\sigma_0\ne 0$ and $\sigma_1\in \mathbb{R}$. Let us assume  without restriction that $\sigma_0>0$.
We also assume that 
the parameters  appearing in \eqref{sys_intro}   satisfy the following hypotheses. 

\begin{Hypo} 
\begin{enumerate}
    \item $\displaystyle \sigma_1\in \left] -\sigma_0,\ \sigma_0  \right[,$
	\item $\lambda_1$, $\lambda_2$, $\lambda_3$ are positive constants,
	\item $\theta\in ]-\pi,\pi[\setminus \left\lbrace \pm \dfrac{\pi}{2},\arctan\left(\dfrac{2\lambda_3}{2\lambda_1-\lambda_2} \right)  \right\rbrace$. \end{enumerate} 
	\label{hypo_bruit}
\end{Hypo}

Note that under Assumption \ref{hypo_bruit}.1, the diffusion coefficient doesn't cancel  between the lines $x=-1$ and $x=1$.

\vspace{0.5cm}
Let us introduce the infinitesimal generator associated with the solution of  \eqref{sys_intro}. For $f\in \mathcal{C}^2(\mathbb{R}^2,\mathbb{R})$ and $(x,y)\in \mathbb{R}^2$, we define
\begin{eqnarray} \label{gen_eds}
			\mathcal{L}^\varepsilon f(x,y)&=&\displaystyle \lambda_1x(1-x^2)\frac{\partial f}{\partial x}-\left(\lambda_2 y-\lambda_3x(1-x^2) \right)\frac{\partial f}{\partial y}+\frac{\varepsilon^2(\sigma_0+\sigma_1 x)^2}{2}\left( \cos^2(\theta)\frac{\partial^2 f}{\partial x^2}+\sin^2(\theta) \frac{\partial^2 f}{\partial y^2}\right)\nonumber\\
		&&\displaystyle+ \varepsilon^2(\sigma_0+\sigma_1 x)^2\sin(\theta)\cos(\theta)\frac{\partial^2 f}{\partial x\partial y}.
\end{eqnarray}

\begin{Lem}[Lyapunov condition]
\label{pro_lya}

    Under Hypothesis \ref{hypo_bruit}, there exist $\varepsilon_0$ and $\alpha>0$ such that the infinitesimal generator $\mathcal{L}^\varepsilon$ satisfies a Lyapunov condition for $W(x,y)=1+x^4+ \alpha y^2$,  uniformly for $\varepsilon\le \varepsilon_0$: there exist positive constants $\alpha_1$ and $\alpha_2$, independent of $\varepsilon\le \varepsilon_0$ such that 
	\begin{equation}
		\mathcal{L}^\varepsilon W\leq \alpha_1-\alpha_2 W.
		\label{lya}
	\end{equation} 
	\end{Lem}

\begin{proof}
	For all $0<\varepsilon\leq \varepsilon_0$, we have
	\begin{equation*}
		\begin{array}{ll}
			\mathcal{L}^\varepsilon W(x,y) &= 4\lambda_1x^4(1-x^2)- 2\alpha {y} (\lambda_2 y-\lambda_3x(1-x^2)) +\varepsilon^2(\sigma_0+\sigma_1 x)^2 (6 x^2\cos^2(\theta) + \alpha\sin^2(\theta) ) . \\
		\end{array}
	\end{equation*}
The dominant term is
\begin{eqnarray*} && - 4\lambda_1 x^6 - 2\alpha \lambda_2 y^2 -2\alpha   \lambda_3 yx^3\\
&=& - 2 x^6\Big( 2 \lambda_1 + \alpha \lambda_2 \big(\frac{y}{x^3}\big)^2+ \alpha   \lambda_3 \big(\frac{y}{x^3}\big)\Big).
\end{eqnarray*}
We can observe that under Hypothesis \ref{hypo_bruit}.2,  the polynomial $t \longrightarrow  2 \lambda_1 + \alpha \lambda_2 t^2  + \alpha   \lambda_3  t$  will be always positive if and only if its discriminant is negative, i.e. for any $\alpha$ satisfying 
$$\alpha \leq 8 \frac{\lambda_{2}\lambda_1}{ \lambda_3^2}.$$
Therefore, under Hypothesis \ref{hypo_bruit}.2 and for
$\varepsilon\leq \varepsilon_0$, one can conclude that for
$\alpha\in]0,8\lambda_{1}\lambda_{2}/\lambda_{3}^{2}[$, 
\[\mathcal{L}^\varepsilon W\leq \alpha_1-\alpha_2 W,\]
for $\alpha_1$ and $\alpha_2$ (independent of $\varepsilon$), conveniently chosen. 
\end{proof}

It is then standard, using \eqref{lya}, to prove that for any $T>0$ and any fixed $\varepsilon>0$ and  if $\mathbb{E}(W(X^\varepsilon_0,Y^\varepsilon_0))<+\infty$,  there exists a unique solution $((X^\varepsilon_t,Y^\varepsilon_t), t\in [0,T])$ of  \eqref{sys_intro} in the space of continuous processes with  $\mathbb{E}(\sup_{{t\le T}}W(X^\varepsilon_t,Y^\varepsilon_t))<+\infty$. Moreover \eqref{lya} also implies that if $\sup_{\varepsilon\le \epsilon_{0}}\mathbb{E}(W(X^\varepsilon_0,Y^\varepsilon_0))<+\infty$, then   $\sup_{\varepsilon\le \epsilon_{0}}\mathbb{E}(\sup_{{t\le T}}W(X^\varepsilon_t,Y^\varepsilon_t))<+\infty$. 
In all what follows, we will assume that 
$$\sup_{\varepsilon\le \epsilon_{0}}\mathbb{E}(W(X^\varepsilon_0,Y^\varepsilon_0))<+\infty$$
and we will denote by $\ (P^\varepsilon_t)_{t\ge 0}\ $ the semigroup associated with $((X^\varepsilon_t,Y^\varepsilon_t), t\in \mathbb{R}_{+})$. 

\medskip

Assuming in addition  that $(X^\varepsilon_0,Y^\varepsilon_0)$ converges in probability to the deterministic point $(x_{0},y_{0})$, classical arguments  allow to show that for any $T>0$, the processes   $((X^\varepsilon_t,Y^\varepsilon_t), t\in [0,T])$ converge in probability to the solution $((x_{t},y_{t}),  t\in [0,T])$ of the  deterministic system \eqref{sys_main} issued from $(x_{0}, y_{0})$, when $\varepsilon$ tends to 0. 
System \eqref{sys_main} admits three equilibria defined by
\begin{equation}
	z_1=(-1,0), \ z_2=(0,0)\text{ and }z_3=(1,0).
	\label{pt_eq_zi}
\end{equation}

 Let us note that $z_2$ is a saddle-node and $z_1$ and $z_3$ are attractive points. The basin of attraction of $z_{1}$ (resp. $z_{3}$) is the open left (resp. right) half plane. The  basin of attraction of $z_{2}$ is the $y$-axis.\\

Making  $\varepsilon$ tend to 0 before time $t$ tend to infinity yields
\begin{equation}
		\underset{t\to +\infty}{\lim}\ \underset{\varepsilon\to 0}{\lim}(X^\varepsilon_t,Y^\varepsilon_t)=z_1\mathbb{1}_{\left\lbrace x_0<0\right\rbrace }+z_2\mathbb{1}_{\left\lbrace x_0=0\right\rbrace }+z_3\mathbb{1}_{\left\lbrace x_0>0\right\rbrace }.
\end{equation}

 We are now interested in the limit $t\rightarrow +\infty$ first, when $\epsilon $ is fixed. In that aim, 
let us associate with the Lyapounov function $W$  the norm $\Vert\cdot \Vert_W$ defined on the set of finite signed measures  by 
\begin{equation}
	\Vert\mu\Vert_W=\int_{\mathbb{R}^2} W(x,y)|\mu|(dx,dy),
	\label{norme_lya_annexe}
\end{equation}
 $\mu$ being a finite signed measure on $\mathbb{R}^2$.\\

We can state the two main results of the paper.

\begin{The}
	Let $\varepsilon\in ]0,\varepsilon_0[$ fixed. Under  Hypothesis \ref{hypo_bruit}, the process $(X^\varepsilon_t,Y^\varepsilon_t)_{t\geq 0}$ solution of \eqref{sys_intro} admits a unique invariant probability measure $\pi^\varepsilon$ with  Lesbegue density. There exist a constant $C>0$ and $\gamma>0$ such that for any $(x,y)\in\mathbb{R}^2$,
	\begin{equation}
		\Vert P^\varepsilon_t((x,y),.)-\pi^\varepsilon(.) \Vert_W \leq C W(x,y) e^{-\gamma t}.
	\end{equation}

	\label{thm_mes_inv_unique}
\end{The}
The proof of Theorem \ref{thm_mes_inv_unique} will be postponed to Section \ref{section_exist_inv_mes}.
 
\begin{The}
	When $\varepsilon$ tends to 0, the sequence $(\pi^\varepsilon)_{\varepsilon<\varepsilon_0}$ converges to the measure $\pi$ defined by 
	\begin{equation}
		\pi=\left\lbrace \begin{array}{ll}
			\delta_{(1,0)}& if\  \sigma_1<0,\\
			&\\
			\displaystyle \frac{1}{2}\delta_{(1,0)}+\frac{1}{2}\delta_{(-1,0)}& if\  \sigma_1=0,\\
			&\\
			\delta_{(-1,0)}& if\  \sigma_1>0.\\
		\end{array}
		\right.
	\end{equation} 
	\label{thm_mes_lim_exp}
\end{The}
This result shows that the most stable equilibrium point is the one with the associated diffusion coefficient closest to zero.

The proof of Theorem \ref{thm_mes_lim_exp} will be postponed to Section \ref{section_limit_supp_inv_mes}.

\section{Proof of Theorem \ref{thm_mes_inv_unique}}\label{section_exist_inv_mes}

The proof of Theorem \ref{thm_mes_inv_unique} is based on Proposition 6.32, Theorems 6.34 and Theorem 8.15 in  \cite{benaim2022markov} (see also  Corollary 7.2 in \cite{bellet2006ergodic}). It will consist in proving  first  (Lemma  \ref{lem:hormander})
that the stable equilibria $z_1$ and $z_3$ of \eqref{sys_main} satisfy the weak H\"ormander condition and second (Lemma \ref{lem:access})  that these two points are accessible. As explained in the introduction, the difficulty in this proof comes from the degeneracy of the stochasticity. \\

\begin{Lem}
     Under Hypothesis \ref{hypo_bruit}, the stable equilibria points $z_1$ and $z_3$ defined by \eqref{pt_eq_zi} satisfy the weak H\"ormander condition.
    \label{lem:hormander}
\end{Lem}

\begin{proof}	
	Let us introduce the vector fields $F_0=\lambda_1x(1-x^2)\partial_x+(\lambda_3x(1-x^2)-\lambda_2 y) \partial_y$ and $F_1=\varepsilon (\sigma_{0} +\sigma_{1}x) \left( \cos(\theta)\partial_{x}+\sin(\theta) \partial_{y}\right) $.  The infinitesimal generator \eqref{gen_eds} can be written as : 
	\begin{equation}
		\mathcal{L}^\varepsilon= F_0+\frac{1}{2}F^2_1,
	\end{equation}
	where $F^2_1= \varepsilon^2(\sigma_{0} +\sigma_{1}x)^2( \cos^2(\theta)\partial_{xx}+\sin^2(\theta) \partial_{yy}+2\cos(\theta)\sin(\theta)\partial_{xy}).$
	\\
	We introduce the sequence of vector fields $(\mathcal{A}_k)$, $k\geq0$, defined by $\mathcal{A}_0=\{F_0, F_{1}\}$ and $\mathcal{A}_{k+1}=\mathcal{A}_{k}\cup\{ [F_i,G],G\in\mathcal{A}_{k} , i=0,1\}$ where $[U,V]$ is  the Lie's bracket of $U$ and $V$. We denote by  $\mathcal{A}_k(x,y)$ the vector space spanned by the vector fields of $\mathcal{A}_k$ for $(x,y)\in \mathbb{R}^2$.
	
As can be seen in \cite[Section 6.5.2]{benaim2022markov} (see also \cite{hormander1967hypoelliptic}), a point $(x,y)\in \mathbb{R}^2$ satisfies the weak H\"ormander's condition  if there exists $k>0$ such that the vectoriel space $\mathcal{A}_k(x,y)$ spans $\mathbb{R}^2$.\\


For $k=0$, $\mathcal{A}_0(z_{1})$ and $\mathcal{A}_0(z_{3})$ only span  $\displaystyle \mathbb{R}$, since $F_{0}$ is nul in this case. 
For $k=1$, 
we compute 
	\begin{eqnarray*}
					\left[ F_0, F_1 \right](x,y) &=& F_0(F_1)-F_1(F_0)\\ 
		 &=& \varepsilon\cos(\theta)\left(\lambda_1(\sigma_{0} +\sigma_{1}x)(3x^2-1)+ \lambda_1\sigma_1x(1-x^2)\right) \partial_x\\
   &&+\varepsilon\left[\cos(\theta)\lambda_3(\sigma_{0} +\sigma_{1}x)(3x^2-1)+\sin(\theta)\left((\sigma_{0} +\sigma_{1}x)\lambda_2+\lambda_1\sigma_1x(1-x^2) \right)\right]\partial_y. 
	\end{eqnarray*}	
	
	Then  $$\left[ F_0, F_1 \right](z_{1}) = \varepsilon  (\sigma_{0} -\sigma_{1}) \Big(2  \lambda_1 \cos(\theta)\partial_x + \big( 2\lambda_3 \cos(\theta)+ \lambda_2 \sin(\theta) \big)\partial_y\Big),$$
and 	$\left[ F_0, F_1 \right](z_{3})$ has a similar value with  $ (\sigma_{0} -\sigma_{1})$ replaced by  $(\sigma_{0} +\sigma_{1})$.
Let us study the colinearity of  $F_1$ and $\left[ F_0, F_1 \right]$ at $z_{1}$. 
The matrix 
	\begin{equation} 
		\begin{pmatrix}
			2   \lambda_1 \cos(\theta)& \cos(\theta)\\
			2\lambda_3 \cos(\theta) + \lambda_2  \sin(\theta) & \sin(\theta)\\
		\end{pmatrix}
		\label{colineaire_mat}
	\end{equation} 
has the determinant  $\   \cos(\theta) \big( (2   \lambda_1- \lambda_2)  \sin(\theta) - 2\lambda_3 \cos(\theta)  \big)\ $
which is non zero under   Hypothesis \ref{hypo_bruit}.3.
  Therefore,  $\mathcal{A}_1(z_{1})$ spans $\mathbb{R}^2$ and the same property holds for $z_{3}$. Thus the two stable equilibria $z_{1}$ and $z_{3}$  satisfy  the weak H\"ormander condition. 
\end{proof}

Let us now prove their accessibility.

\begin{Lem}
     Under Hypothesis \ref{hypo_bruit},  the stable equilibria $z_1$ and $z_3$ are accessible for $(P^\varepsilon_t)$ from any $(x_{0},y_{0})\neq (0,0)$.
    \label{lem:access}
\end{Lem}

The proof will be based on Proposition 6.32 in \cite{benaim2022markov}. 
In that aim, we introduce the control system associated to the solution $\left( X_t^\varepsilon ,Y_t^\varepsilon \right)$ of  \eqref{sys_intro}, defined by:

\begin{equation}
	\left\{
	\begin{aligned}
		dx^{\varphi}_t&= \lambda_1x^{\varphi}_t(1-(x^{\varphi}_t)^2)dt+\varepsilon (\sigma_{0}+ \sigma_{1}x) \cos(\theta)\varphi(t)dt \\ 
		dy^{\varphi}_t&= -\lambda_2 y^{\varphi}_t+\lambda_3x^{\varphi}_t(1-(x^{\varphi}_t)^2)dt+\varepsilon (\sigma_{0}+ \sigma_{1}x) \sin(\theta)\varphi(t)dt, \\
	\end{aligned}
	\right.  
	\label{syst_controle_chp3}
\end{equation}
where the control function $\varphi$ is a piecewise continuous  function defined on $\mathbb{R}_+$.\\
 
\begin{proof}
    Using Proposition 6.32 from \cite{benaim2022markov}, a point $z\in \mathbb{R}^2$ is accessible if and only if  for every neighborhood $U$ of $z$, there exist a control function $\varphi$ and a time $t\geq 0$ such that $\left(x^{\varphi}_t,y^{\varphi}_t\right)\in U$.
    We will detail the proof for one of the two equilibria. Let's show that $z_1$ is accessible. 
       If the initial condition is taken in $z_1$'s basin of attraction, i.e. $x_0<0$, then  taking $\varphi\equiv 0$ makes $\left(x^{\varphi},y^{\varphi}\right)$ follow the flow and converge towards $z_1$. \\
    If the initial condition is taken in $z_3$'s basin of attraction, i.e. $x_0>0$, then similarly the flow \eqref{sys_main} converges towards $z_3$. Thus we can assume without loss of generality that  $x_0= 1+\delta$, where $\delta\neq 0$ will be conveniently chosen. We are looking for a control function $\varphi$ allowing the flow \eqref{syst_controle_chp3}  to attain the attraction basin of $z_{1}$.     Since by Hypothesis \ref{hypo_bruit}, $|\sigma_{1}|<\sigma_{0}$, there exists $\delta\neq 0$ such that for any $x\in [-1/2;1+\delta]$, $|\sigma_{0}+\sigma_{1}x|>0$. We fix the parameter $\delta$ and choose a real value $k$ such that for any $x\in [-1/2;1+\delta]$, 
    $$ \lambda_1x(1- x^2)- k \varepsilon (\sigma_{0}+ \sigma_{1}x) \cos(\theta)<-1.$$
   For example,
   any $$k> \frac{1+\lambda_{1}}{\Big(\sigma_{0}-|\sigma_{1}|(1+\delta)\Big)\cos\theta \varepsilon
    }$$
   is suitable.
   Taking the constant control  $\,\varphi(t)= -k$, one  obtains that starting from $(1+\delta,y)$ (for $y\in \mathbb{R}$), $dx^{\varphi}_t<0$ as soon as $x(t)\in [-1/2;1+\delta]$.  Therefore after a certain amount of time,
   $x(t)$ will become less than $-1/2$. In particular,  there exists $t_1>0$ such that $\left(x^{\varphi}_{t_1},y^{\varphi}_{t_1}\right)\in \mathbb{R}^-_{*}\times \mathbb{R}$. Once the dynamics is in $z_1$'s basin of attraction,  the flow will make it converge to $z_{1}$.\\
If $x_{0}=0$ and $y_{0}\neq 0$, then the dynamics of \eqref{sys_dyn}  (i.e. $\varphi=0$) will lead the system to be in the open half plan $\mathbb{R}^+_{*} \times \mathbb{R}$ or $\mathbb{R}^-_{*} \times \mathbb{R}$. We are back to one of the previous situations.

\end{proof}

\begin{proof}[Proof of the Theorem \ref{thm_mes_inv_unique}]
    Using the lemmas \ref{lem:access}, 
    \ref{lem:hormander} and \ref{pro_lya}, we can use the Theorems 6.34 and 8.15 from \cite{benaim2022markov}, which concludes the proof.
\end{proof}

\section{Limit and support of the invariant probability measure when $\varepsilon\to 0$}\label{section_limit_supp_inv_mes}

In this section, we explain the different steps of the proof of  Theorem \ref{thm_mes_lim_exp}. First, we show the tightness of $\left(\pi^\varepsilon\right)_{\varepsilon\le \varepsilon_0}$ and we identify the limiting values of $\left(\pi^\varepsilon\right)_{\varepsilon\le \varepsilon_0}$. In a second time, we adapt the work of Freidlin and Wentzel in \cite{freidlin1998random} in our degeneracy case to prove the uniqueness of the limit and to characterize it.\\ 

\subsection{Existence and characterization of the limiting values of $\left(\pi^\varepsilon\right)_{\varepsilon\le \varepsilon_0}$}

\begin{Lem}
\label{tight}
    The sequence $\left(\pi^\varepsilon\right)_{\varepsilon\le \varepsilon_0}$ is tight.
\end{Lem}

\begin{proof}
Let us denote by $\mathfrak{B}(R))$ the ball of  $\mathbb{R}^2$ centered at $0$ with radius $R$. 
In order to show that the sequence $\left(\pi^\varepsilon\right)_{\varepsilon\le \varepsilon_0}$ is tight, we introduce the occupation measure $r^\varepsilon$ defined for any $(x_0,y_0)\in \mathbb{R}^2$ and $t\in \mathbb{R}_{+}$ and $A$ a Borel subset of $\mathbb{R}^2$ by 
 $$r^\varepsilon(t,(x_0,y_0),A) = \displaystyle \frac{1}{t}\int^t_0 P^{\varepsilon}_s\left( \mathbb{1}_{A }\right)(x_0,y_0)ds=\displaystyle\frac{1}{t}\int^t_0 \mathbb{E}_{(x_0,y_0)}\left( \mathbb{1}_{A}(X_s^\varepsilon,Y_s^\varepsilon)\right) ds.$$
 Then
\begin{equation}
\begin{array}{rl}
	r^\varepsilon(t,(x_0,y_0),\mathbb{R}^2\setminus \mathfrak{B}(R))	=&\displaystyle\frac{1}{t}\int^t_0 \mathbb{E}_{(x_0,y_0)}\left( \mathbb{1}_{\left\lbrace  \mathbb{R}^2\setminus\mathfrak{B}(R)\right\rbrace }(X_s^\varepsilon,Y_s^\varepsilon)\right) ds.\\
\end{array}
\label{mesure_occupation_eds}
\end{equation}

Let us prove  that $\forall\ \delta>0,\exists\ R>0$ such that $\forall\ t>0,\ \forall\  \varepsilon\in ]0,\varepsilon_0]$,

\begin{equation*}
r^\varepsilon(t,(x_0,y_0),\mathbb{R}^2\setminus \mathfrak{B}(R))\leq \delta.
\label{mesure_occ_tension_pi_eps}
\end{equation*}

We introduce for any $\varepsilon\le \varepsilon_0$  the sequence of stopping times $(\tau_n^{\varepsilon})_{n}$ defined by $\tau^{\varepsilon}_n=\inf\{t\geq 0 :\Vert\left(X^{\varepsilon}_t, Y^{\varepsilon}_t \right)\Vert\geq n\}$. The initial condition of the process being the fixed vector $(x_{0},y_{0})$, we have seen that Section 2 that $\sup_{\varepsilon\le \epsilon_{0}}\mathbb{E}(\sup_{{t\le T}}W(X^\varepsilon_t,Y^\varepsilon_t))<+\infty$. It is then easy to prove that for any $\varepsilon$, the sequence $(\tau^\varepsilon_{n})_{n}$ tends almost surely to infinity. We have  that 
\begin{equation}
\mathbb{E}_{(x_0,y_0)}\Big(W\left( X^\varepsilon_{t\wedge \tau_n},Y^\varepsilon_{t\wedge \tau_n}\right)\Big) = W(x_0,y_0)+\mathbb{E}_{(x_0,y_0)}\left(\int_0^{t\wedge \tau_n} \mathcal{L}^{\varepsilon} W\left( X^\varepsilon_s,Y^\varepsilon_s\right) ds \right).
\label{preuve_tension_ito}
\end{equation}
Note that  the function $(x,y)\in \mathbb{R}^2 \to W(x,y)$ is continuous and lower bounded by $m_{R}>0$ on the complementary of the open ball of radius $R$.  Note also that $m_{R}$ tends to infinity with $R$. Therefore, using Lemma \ref{pro_lya}, we obtain 
\begin{equation*}
	\displaystyle \mathbb{E}_{(x_0,y_0)}\Big(W\left( X^\varepsilon_{t\wedge \tau_n},Y^\varepsilon_{t\wedge \tau_n}\right) \Big)\leq W(x_0,y_0)+ \alpha_1\ \displaystyle \mathbb{E}_{(x_0,y_0)}(t\wedge \tau_n)\displaystyle-\alpha_2 \,m_{R}\,\mathbb{E}_{(x_0,y_0)}\left(\int_0^{t\wedge \tau_n} \mathbb{1}_{\left\lbrace ||(X^\varepsilon_s,Y^\varepsilon_s)||>R \right\rbrace} ds \right),
	\end{equation*}
	from which we deduce that 
	\begin{eqnarray*}
	\displaystyle \alpha_2 \,m_{R}\,\mathbb{E}_{(x_0,y_0)}\left(\int_0^{t\wedge \tau_n} \mathbb{1}_{\left\lbrace ||(X^\varepsilon_s,Y^\varepsilon_s)||>R \right\rbrace} ds \right) &\leq& W(x_0,y_0)+\alpha_1\ \displaystyle \mathbb{E}_{(x_0,y_0)}(t\wedge \tau_n),\end{eqnarray*} 
 since $\,W\,$ is positive.
Finally making  $n$ tend to infinity, we obtain that
\begin{equation}
r^\varepsilon\left( t,(x_0,y_0),\mathbb{R}^2\setminus \mathfrak{B}(R) \right)\leq \frac{t\alpha_1+W(x_0,y_0)}{ t\alpha_2  \,m_{R}}.
\label{preuve_tension_maj}
\end{equation}
Using Theorem \ref{thm_mes_inv_unique}, we have
$$\pi^\varepsilon(\mathbb{R}^2\setminus \mathfrak{B}(R) )\leq r^\varepsilon\left( t,(x_0,y_0),\mathbb{R}^2\diagdown \mathfrak{B}(R) \right) + CW(x_{0},y_{0}) \frac{1}{\gamma t} (1-e^{-\gamma t}).$$
The rhs term can be made as small as required for $t$ and $R$ large
enough, uniformly in $\epsilon$. 
Therefore the sequence $\left(\pi^\varepsilon\right)_{\varepsilon\le \varepsilon_0}$ is tight.

 \end{proof}

\begin{Lem}
    The limiting values of $\left(\pi^\varepsilon\right)_{\varepsilon\le \varepsilon_0}$  have the form
    \begin{equation}
		\sum_{i=1}^{3} a_i \delta_{z_i},
	\end{equation}
	where ${z_i}$, $i\in \{1,2,3\}$ are defined in \eqref{pt_eq_zi} and the coefficients $a_i\geq0$ satisfy $\displaystyle\sum_{i=1}^3 a_i=1.$ 
 \label{lem_cara_pi}
\end{Lem}

\begin{proof}
We introduce the operator $\mathcal{L}^0$ defined for all $f\in \mathcal{C}^\infty_c(\mathbb{R}^2)$ by 
\begin{equation}
\mathcal{L}^0f(x,y)=\lambda_1 x(1-x^2)\tfrac{\partial f}{\partial x}(x,y)-\left(\lambda_2 y-\lambda_3 x(1-x^2)\right)\tfrac{\partial f}{\partial y}(x,y).
\label{gen_lim}
\end{equation} 

Let us consider an accumulation point  $\pi$ of the sequence $\left(\pi^\varepsilon\right)_{\varepsilon\le\varepsilon_0}$. We will show that \begin{equation}
\int_{\mathbb{R}^2} \mathcal{L}^0f(x,y) \pi(dx,dy)=0, \quad \forall f\in \mathcal{C}^\infty_c(\mathbb{R}^2).
\label{condition_mes_lim}
\end{equation}

\vspace{0.5cm}

Let be $f\in \mathcal{C}^\infty_c(\mathbb{R}^2)$. We decompose $\displaystyle \int_{\mathbb{R}^2} \mathcal{L}^0f(x,y) \pi(dx,dy)$ as follows :\\
\begin{equation*}
\begin{array}{rcl}
	\displaystyle \int_{\mathbb{R}^2} \mathcal{L}^0f(x,y) \pi(dx,dy)& = &\displaystyle \int_{\mathbb{R}^2} \mathcal{L}^0f(x,y) \big( \pi(dx,dy)-\pi^{\varepsilon_k}(dx,dy)\big)\\
	\\
	&& \displaystyle + \int_{\mathbb{R}^2} \mathcal{L}^{\varepsilon_k} f(x,y) \pi^{\varepsilon_k}(dx,dy)\\
	\\
	&&\displaystyle+ \int_{\mathbb{R}^2} \left( \mathcal{L}^0 -\mathcal{L}^{\varepsilon_k}\right) f(x,y) \pi^{\varepsilon_k}(dx,dy).  \\
\end{array} 
\label{diff_mes_inv}
\end{equation*}

The sequence $(\pi^{\varepsilon_k})_{\varepsilon_k<\varepsilon_0}$ converges weakly to $\pi$ and the first term in the rhs of previous equation tends to 0 when $\varepsilon_k\underset{k\to +\infty}{\longrightarrow} 0$. The second term is null as $\pi^{\varepsilon_k}$ is an invariant measure for the diffusion operator $ \mathcal{L}^{\varepsilon_k}$. \\
Using a majoration of the third term, we obtain 
\begin{equation*}
\begin{array}{lcl}
\left|	\displaystyle \int_{\mathbb{R}^2} \left(\mathcal{L}^{\varepsilon_k} f(x,y)-\mathcal{L}^0f(x,y)\right) \pi^{\varepsilon_k} (dx,dy) \right|&\leq& \displaystyle \frac{\varepsilon_k^2}{2} \int_{\mathbb{R}^2} (\sigma_{0}+\sigma_{1}x)^2\left| \frac{\partial^2 f(x,y)}{\partial x^2}+\frac{\partial^2 f(x,y)}{\partial y^2} \right| \pi^{\varepsilon_k} (dx,dy) \\
	&&\\ 
	&\leq&\displaystyle  C_f\ \varepsilon_k^2,\\		
\end{array}
\end{equation*}
where $\displaystyle C_f=\frac{1}{2}\underset{(x,y)\in\mathbb{R}^2}{\sup} \left((\sigma_{0}+\sigma_{1}x)^2\left| \frac{\partial^2 f(x,y)}{\partial x^2}+\frac{\partial^2 f(x,y)}{\partial y^2} \right| \right)<+\infty$ since $f$ has a compact support. \\
That concludes the proof of \eqref{condition_mes_lim}. \\
 
The operator $\mathcal{L}^0$ defined by \eqref{gen_lim} is associated with the deterministic system \eqref{sys_dyn}. We have for all $f\in \mathcal{C}^\infty_c(\mathbb{R}^2)$
\begin{equation*}
f(x_t,y_t)=f(x_0,y_0)+\int_0^t \mathcal{L}^0f(x_s,y_s) ds, 
\end{equation*}
where $(x_0,y_0)\in \mathbb{R}^2$ is the initial condition.

Integrating the initial condition with respect to the measure $\pi$, we obtain  
\begin{equation*}
\int_{\mathbb{R}^2} f(x_t,y_t)\pi(dx_0,dy_0)=\int_{\mathbb{R}^2} f(x_0,y_0)\pi(dx_0,dy_0)+\int_{\mathbb{R}^2} \int_0^t \mathcal{L}^0f(x_s,y_s) ds\ \pi(dx_0,dy_0),
\end{equation*}
where the last term vanishes because $\pi$ is an invariant measure. 


Taking the limit  when $t\to +\infty$, using the different basins of attraction and Lebesgue's theorem, we obtain that
\begin{align*}
& \int_{\mathbb{R}^2} f(x_0,y_0)\pi(dx_0,dy_0)  \\
& =\int_{\mathbb{R}^2} f\left(z_1\mathbb{1}_{(x_0<0)}+ z_2\mathbb{1}_{(x_0=0)}+z_3\mathbb{1}_{(x_0>0)}\right) \pi(dx_0,dy_0)\\
&= f(z_{1})\pi(\mathbb{R}_{-}\times \mathbb{R}) + f(z_{2})\pi(\{0\}\times \mathbb{R}) + f(z_{3})\pi(\mathbb{R}_{+}\times \mathbb{R}),
\label{support_mes_inv}
\end{align*}
which concludes the proof.
\end{proof}

\subsection{Uniqueness and characterization of the limiting values of $\left(\pi^\varepsilon\right)_{\varepsilon\le \varepsilon_0}$}

We will adapt in our degeneracy framework the method introduced in \cite{freidlin1998random}. Here the diffusion matrix of System \eqref{sys_intro} is not invertible. Our strategy will be uniquely based on the first component of the system. Indeed the latter is the principal component of  the stochastic differential system \eqref{sys_intro}, the drift of the second component pushing it toward $0$.
The equilibria of the associated dynamical system \eqref{sys_main} are aligned on the horizontal axis and the cost  of trajectories will be reduced to the movement on the horizontal axis, allowing a reduction of the two-dimensional problem to a one-dimensional study. 

The diffusion matrix is given by
$$
\sigma(x,y) = (\sigma_{0}+\sigma_{1}x)\left(\begin{array}{cc}\cos \theta&0\\\sin \theta&0\end{array}\right).
$$
Using   Theorem 5.6.7 in \cite{dembozeitouni} and the specific form of the drift $(b_{1},b_{2})$, we obtain that the action functional is given by
\begin{eqnarray*}I_{(x_{0},y_{0})}(f_{1}, f_{2}) =&& \inf\Big\{\int_{0}^T |\dot g(s)|^2ds\ ;\ g_{1}\in H_{1},  f_{1}(t) = x_{0}+ \int_{0}^t b_{1}(f_{1}(s))ds + \int_{0}^t \sigma_{11}(f_{1}(s))g_{1}(s)ds \ \hbox{ and } \\
&& \  f_{2}(t) = y_{0}+ \int_{0}^t b_{2}(f_{1}(s),f_{2}(s))ds + \int_{0}^t \sigma_{21}(f_{1}(s))g_{1}(s)ds\Big\},
\end{eqnarray*}
and $I_{(x_{0},y_{0})}(f_{1}, f_{2}) =+\infty$ otherwise. This minimization problem is solved by considering first the constraint on the first coordinate, and then by seeing if the second constraint is satisfied or not. Minimization under the first constraint is known to be equivalent to minimize the Freidlin-Wentzell functional (see \cite{dembozeitouni}, remark p. 214). If the second constraint is not satisfied for the minimizer, the action functional will be taken to be $+\infty$. 

Therefore, as seen in \cite{freidlin1998random} (see also \cite{dembozeitouni} Theorem 5.6.12), the behavior of probabilites of
large deviations from the most "probable" trajectory in $x$, i.e. the trajectory of the dynamical system \eqref{sys_main}, can be described thanks to the action functional  $S^{\sigma}_{0T}$ acting on smooth functions $w$ from $[0,T]$ to $\mathbb{R}$, and defined for   $\varepsilon\in ]0,\varepsilon_0]$ and $\theta\in ]-\pi,\pi[\diagdown \left\lbrace \pm\frac{\pi}{2} \right\rbrace$ by
\begin{equation}
S^{\sigma}_{0T}((w_t)_{t\geq 0}):=\frac{1}{(\varepsilon\cos(\theta))^2}\int_0^T L^{\sigma}(w_s,\dot{w}_s)ds,
\label{fct_action}
\end{equation}
where the lagrangien $L^{\sigma}(t,w_t,\dot{w}_t)$ is defined as 
\begin{equation*}
L^{\sigma}(w_{t},\dot{w}_{t})=\frac{1}{2}\left(\frac{\dot{w}_{t}-\lambda_{1}w_{t}(1-w_t^2)}{\sigma_{0}+\sigma_{1} w_t}\right)^2.
\end{equation*}

Let introduce a parameter $\displaystyle \delta \in \left]0,1/2\right[$. We consider three  disjoints  subsets of  $\mathbb{R}^2$,  $K^\delta_1$, $K^\delta_2$ and $K^\delta_3$ defined by 
\begin{equation}
K^\delta_i=\{(x,y)\in \mathbb{R}^2, | x-x_i|^2\leq \delta^2\}, 
\end{equation}
where $x_i$, $i\in \left\lbrace 1,2,3 \right\rbrace $ is the first component of  $z_i$ defined in \eqref{pt_eq_zi}. 
Note that these sets are bands and not compact as in
\cite{freidlin1998random}, however their basis in $x$ are compact
intervals.  Due to the particular form of the system
\eqref{sys_intro}, entrances and exits in these bands are described by
the $x$ coordinate only which obeys a stochastic differential equation
independent of $y$. 
Let us fix $T>0$. \\

 An essential role will be played by the function characterizing the difficulty of passage from $K_i^\delta$ to a  $K_j^\delta$, defined by
\begin{equation*}
	V^{\sigma}_{ij}:=\inf \left\lbrace S^{\sigma}_{0T}((w_t)_{t\geq 0});  (w_0,0)\in K_i^\delta,\ (w_T,0)\in K_j^\delta \right\rbrace .
\end{equation*}
The global passage cost  to  $K_j^\delta$ is the quantity
$$W^\sigma(K_j^\delta) = \min_{g\in G(i)}\sum_{(m\to n )\in g}V_{mn}^\sigma,$$
where for $i\in \{1,2,3\}$, $G(i)$ is the set of graphes from $\{1,2,3\}\setminus\{i\}$ to $\{1,2,3\}$ without cycles and such that any point of  $\{1,2,3\}\setminus\{i\}$ is the initial point of exactly one arrow. This quantity will play a main role in the following.

Let us also introduce the cost matrix
$P^\sigma=((V^{\sigma}_{ij}))_{i,j\in\{1,2,3\}}$. Our aim in to
compute this matrix or at least to know which coefficients are
positive. Indeed, if for any $j\neq i$, $V^{\sigma}_{ij}>0$,  the set
$K_i^\delta$ is said stable. In the contrary, this set is  unstable. These properties will be related to the long time behavior of the flow. 

We recall the well known following  property concerning the extrema of the action functional.
\begin{Pro} Let us assume that $\sigma_{1}=0$. 
The extrema of the functional $S^{\sigma_{0}}_{0T}$ are attained by the two flows solution of
$$\dot{w}_{t}=\pm \lambda_{1}w_{t}(1 - w_{t}^2).$$
\end{Pro}

When the sign is $+$,  the extremum is attained by the solution of the dynamical system \eqref{sys_main} and when the sign is $-$, the dynamics realizing the extremum has to go against the flow. 

\begin{proof} Let us give an elementary proof (see also Theorem 3.1 of Chapter 4 in \cite{freidlin1998random}). 
Denoting $F(w) = \lambda_{1} w(1-w^2)$, we have
$$\int_{0}^T (\dot w_{t} -F(w_{t}))^2 dt = \int_{0}^T (\dot w_{t} +F(w_{t}))^2 dt - 4 \int_{0}^T \dot w_{t} F(w_{t}) dt = \int_{0}^T (\dot w_{t} +F(w_{t}))^2 dt - 4 (V(w_{T})  - V(w_{0})),$$
where V is a primitive of $F$.
The result follows.
\end{proof}

\begin{Pro} Let us assume that $\sigma_{1}=0$. 
The solution of $\,\dot{w}_{t}=- \lambda_{1}w_{t}(1 - (w_{t})^2)\,$ is given by $x^\varphi$, with
$$\varphi(t) =- \frac{2 \lambda_{1} w_{0}(1-w^2_{0}) e^{-\lambda_{1}t}}{\varepsilon\sigma_{0}\cos(\theta)\Big(1-w_0^2+w_0^2e^{-2\lambda_{1}t}\Big)^{3/2}}.$$
We also have in this case
$$S^{\sigma_{0}}_{{0T}}(w)= \frac{\lambda^2_{1}(1-w^2_{0})^2}{2(\varepsilon\sigma_{0}\cos(\theta))^2}\left(\frac{1}{\big(1-w_0^2+w_0^2e^{-2\lambda_{1}T}\big)^{2}}-1\right).$$
\end{Pro}

\begin{proof}
A direct computation gives that
$$w_{t} = w_{0}\frac{e^{-\lambda_{1}t}}{\sqrt{1-w_0^2+w_0^2e^{-2\lambda_{1}t}}} $$
and then
$$\dot{w}_{t} = -\lambda_{1} w_{t}(1 - w_{t}^2) = - \lambda_{1} w_{0}(1-w^2_{0})\frac{e^{-\lambda_{1}t}}{\big(1-w_0^2+w_0^2e^{-2\lambda_{1}t}\big)^{3/2}},$$
from which we deduce $$\varphi(t) = \frac{-2\lambda_{1}w_{t}(1-w^2_{t})}{\varepsilon\sigma_{0}\cos(\theta)}= - \frac{2 \lambda_{1} w_{0}(1-w^2_{0}) e^{-\lambda_{1}t}}{\varepsilon\sigma_{0}\cos(\theta)\big(1-w_0^2+w_0^2e^{-2\lambda_{1}t}\big)^{3/2}}.$$

We are now able to compute explicitely the value of the action
functional on these extrema. Of course, it vanishes  on the solution of \eqref{sys_main}. For $w$ solution of $\dot{w}_{t} = -\lambda_{1} w_{t}(1 - (w_{t})^2)$, 
\begin{eqnarray*}S^{\sigma_{0}}_{{0T}}(w)&=& \frac{1}{2(\varepsilon\sigma_{0}\cos(\theta))^2}\int_{0}^T (-2\lambda_{1} w_{t}(1 - w_{t}^2))^2dt\\
&=& \frac{4\lambda^2_{1}}{2(\varepsilon\sigma_{0}\cos(\theta))^2}\int_{0}^T w_{0}^2(1-w^2_{0})^2\frac{e^{-2\lambda_{1}t}}{\big(1-w_0^2+w_0^2e^{-2\lambda_{1}t}\big)^{3}}dt\\
&=& \frac{\lambda^2_{1}(1-w^2_{0})^2}{2(\varepsilon\sigma_{0}\cos(\theta))^2}\left(\frac{1}{\big(1-w_0^2+w_0^2e^{-2\lambda_{1}T}\big)^{2}}-1\right).
\end{eqnarray*}

\medskip
Note that another way to compute $S^{\sigma_{0}}_{{0T}}$ is as follows: 

\begin{eqnarray}S^{\sigma_{0}}_{{0T}}(w)&=& \frac{1}{2(\varepsilon\sigma_{0}\cos(\theta))^2}\int_{0}^T (-2\lambda_{1} w_{t}(1 - w_{t}^2))^2dt\nonumber\\
&=&  \frac{1}{2(\varepsilon\sigma_{0}\cos(\theta))^2}\int_{0}^T 2\dot{w}_{t}(-2\lambda_{1} w_{t}(1 - w_{t}^2))dt\nonumber\\
&=&  \frac{2\lambda_{1}}{2(\varepsilon\sigma_{0}\cos(\theta))^2}\int_{w_{0}^2}^{w_{T}^2}du(u-1)du\nonumber\\
&=&  \frac{\lambda_{1}}{2(\varepsilon\sigma_{0}\cos(\theta))^2} \Big((w^2_{T}-1)^2 -(w^2_{0}-1)^2\Big).\label{exactvalue}
\end{eqnarray}

\end{proof}

The explicit form of the extrema of $S^{\sigma_{0}}_{0T}$ allows us to
compute or estimate  the passage cost from one  $K_{i}^\delta$ to another one. 
It is immediate to observe that for any $i\in \{1,2, 3\}$, 
$$V^{\sigma_{0}}_{2i}=V^{\sigma_{0}}_{ii}= 0,$$
since the flow \eqref{sys_main} allows to go from $K_{2}^\delta$ to $K_{1}^\delta$ or $K_{3}^\delta$, and of course to go from $K_{i}^\delta$ to $K_{i}^\delta$.

The minimum on $w_{T}\in K_{2}^\delta$ and $w_{0}\in K_{3}^\delta$ in the expression \eqref{exactvalue} is positive since it is attained as   minimum of a positive function on a compact set. Then $$V^{\sigma_{0}}_{32} >0.$$ 
Furthermore following  the notation in \cite{freidlin1998random} p.150,  we also have $\widetilde V^{\sigma_{0}}_{31}= \infty$ and then from the expression in \cite{freidlin1998random} p.152,
$$V^{\sigma_{0}}_{31} = V^{\sigma_{0}}_{32} >0.$$
 Similar arguments yield 
 $$ V^{\sigma_{0}}_{13}= V^{\sigma_{0}}_{12} >0$$
and since the diffusion coefficient $\sigma_{0}$ is constant, we also have by symmetry of the minimization problem that
$$V^{\sigma_{0}}_{32} =V^{\sigma_{0}}_{12}.$$

Therefore, we can state the following lemma. 
\begin{Lem}  Let us assume that $\sigma_{1}=0$. 
    Under Hypothesis \ref{hypo_bruit}, the cost matrix of \eqref{sys_intro} is defined by
\begin{equation}
	P^\sigma=\begin{pmatrix}
		0 & V^{\sigma_{0}}_{12} & V^{\sigma_{0}}_{12}\\
		0 &0  & 0\\
		V^{\sigma_{0}}_{12} & V^{\sigma_{0}}_{12} & 0\\
	\end{pmatrix},
	\label{pro_mat_cout}
\end{equation}
\label{lem_mat_cost}
\end{Lem}

Let us now compute the global passage cost to the  sets $K^\delta_{i}$. 

\begin{Pro}
\label{prop-W}  Let us assume that $\sigma_{1}=0$. 
The values of the global passage cost $W^{\sigma_{0}}(K^\delta_{i})$ are given for $i\in \{1,2,3\}$ by
$$W^{\sigma_{0}}(K^\delta_{1}) = W^{\sigma_{0}}(K^\delta_{3}) = V^{\sigma_{0}}_{12}\ ;\ W^{\sigma_{0}}(K^\delta_{2}) = 2 V^{\sigma_{0}}_{12}.$$
\end{Pro}
 
 \begin{proof}
 Let us compute $W^{\sigma_{0}}(K^\delta_{1})$. Consider the different graphs included in $G(1)$: $g_{1}= \{2\to 1, 3\to 1\}$, $g_{2}= \{2\to 1, 3\to 2\}$, $g_{3}= \{2\to 3, 3\to 1\}$. We have
 $$\sum_{(m\to n)\in g_{1}} V^{\sigma_{0}}_{mn} =0+V^{\sigma_{0}}_{31} \ ;\ \sum_{(m\to n)\in g_{2}} V^{\sigma_{0}}_{mn} =V^{\sigma_{0}}_{32}\ ;\ \sum_{(m\to n)\in g_{3}} V^{\sigma_{0}}_{mn} =V^{\sigma_{0}}_{12}.$$
 Then, it is immediate that $\ W^{\sigma_{0}}(K^\delta_{1}) =V^{\sigma_{0}}_{12}$. 
 \end{proof}
 
 \begin{Pro}  Let us assume that $\sigma_{1}=0$. 
 The sequence $(\pi^\varepsilon)_{\varepsilon\le \varepsilon_{0}}$ converges when $\varepsilon$ tends to $0$ to the probability measure
 $$\frac{1}{2}\big(\delta_{(-1,0)}+\delta_{(1,0)}\big).$$
 \end{Pro}
\begin{proof} Using Theorem 4.1 in \cite{freidlin1998random} for the action functional defined in \eqref{fct_action}, we know that there exists $\gamma>0$ such that $\pi^\varepsilon(K^\delta_{i})$ belongs to the interval $$\bigg[\exp\bigg(-\varepsilon^2\Big(W^{\sigma_{0}}(K^\delta_{i}) -\min_{j\in\{1,2,3\}}W^{\sigma_{0}}(K^\delta_{j}) +\gamma\Big)\bigg), \exp\bigg(-\varepsilon^{-2}\Big(W^{\sigma_{0}}(K^\delta_{i}) -\min_{j\in\{1,2,3\}}W^{\sigma_{0}}(K^\delta_{j}) -\gamma\Big)\bigg)\bigg].$$
When $\varepsilon$ tends to $0$, the accumulation points of $(\pi^\varepsilon)$ will concentrate on the  sets $K^\delta_{i}$ attaining the minimum of $W^{\sigma_{0}}$, i.e. $K^\delta_{1}$ and $ K^\delta_{3}$. Symmetry arguments on the stochastic system \eqref{sys_intro} (with $\sigma_{1}=0$) with respect to the axis $x=0$, yield $\pi^\varepsilon(x<0)=\pi^\varepsilon(x>0)$, and then, $\pi^\varepsilon(K^\delta_{1})= \pi^\varepsilon(K^\delta_{3})$. Lemmas \ref{tight} and \ref{lem_cara_pi}  allow then to conclude that there is a unique limiting value of $(\pi^\varepsilon)$ given by $\frac{1}{2}\big(\delta_{(-1,0)}+\delta_{(1,0)}\big)$.
\end{proof}

We can now consider the general case where $\sigma_{1}\neq 0$ and satisfies Hypothesis \ref{hypo_bruit}.$1$ and prove Theorem \ref{thm_mes_lim_exp}.

\begin{proof}[Proof of Theorem \ref{thm_mes_lim_exp}.]
We will bounded the passage costs from below and above. 
Let $\delta\in ]0,\frac{\sigma_{0}-\sigma_{1}}{\sigma_{1}}[$ and to
fix ideas, assume that $\sigma_{1}>0$ (the other case is similar). We first consider the passage from $K^\delta_{1}$ to  $K^\delta_{2}$. Let $x\in ]-(1+\delta),0[$, then
$$\sigma_{0}-(1+\delta)\sigma_{1}<\sigma_{0}+ \sigma_{1}x<\sigma_{0}.$$
In particular, if $w=(w_{1},w_{2})$ is a trajectory such that for any $t\in[0,T]$, $w_{1}(t)\in ]-(1+\delta),0[$, then
$$
S^{\sigma}_{0T}((w_t)_{t\geq 0})> S^{\sigma_{0}}_{0T}((w_t)_{t\geq 0}),$$ 
and
$$V^{\sigma}_{12}>V^{\sigma_{0}}_{12}.$$
For the cost passage from $K^\delta_{3}$ to  $K^\delta_{2}$, we assume that $w_{1}(t)\in ]0, (1+\delta)[$, leading to 
$$V^{\sigma}_{32}<V^{\sigma_{0}}_{32}= V^{\sigma_{0}}_{12}.$$ Then we deduce that $\, V^{\sigma}_{12}>V^{\sigma}_{32}>0$. By similar arguments, $\sigma_{1}<0$ yields $\, V^{\sigma}_{32}>V^{\sigma}_{12}>0$.
That allows to compute $W^{\sigma}(K^\delta_{i})$ as in the proof of Proposition \ref{prop-W}. That gives
$$W^{\sigma}(K^\delta_{1}) = V^{\sigma}_{32}\, ;\, W^{\sigma}(K^\delta_{2}) = V^{\sigma}_{32}+ V^{\sigma}_{12} \,; \,W^{\sigma}(K^\delta_{3}) = V^{\sigma}_{12}.$$
The quantity $\min_{j\in\{1,2,3\}}W^{\sigma_{0}}(K^\delta_{j})$ depends on the sign of $\sigma_{1}$. \\

Using again \cite[Theorem 4.2]{freidlin1998random} for the action functional defined by \eqref{fct_action}, we obtain that when $\varepsilon$ tends to 0,  any converging subsequence of $(\pi^\varepsilon)$ converges  to a measure $\pi$ which concentrates on $K^\delta_1$ if $\sigma_1<0$ and on $K^\delta_3$ if $\sigma_1>0$ (and on both sets if $\sigma_1=0$). Using Lemma \ref{lem_cara_pi},  we deduce that there is a unique limiting measure, either $\delta_{(-1,0)}$ if $\sigma_1>0$ or $\delta_{(1,0)}$ if $\sigma_1<0$.
\end{proof}

\bigskip {\bf Acknowledgments}: This work has been supported by the Chair Mod\'elisation Math\'ematique et Biodiversit\'e of Veolia - Ecole polytechnique - Museum national d'Histoire naturelle - Fondation X. It is also funded by the European Union (ERC AdG SINGER, 101054787). Views and opinions expressed are however those of the author(s) only and do not necessarily reflect those of the European Union or the European Research Council. Neither the European Union nor the granting authority can be held responsible for them.


\begin{thebibliography}{99}


  
  \bibitem{benaim2022markov}
  {M. Bena{\"\i}m and T. Hurth},
  {\it Markov Chains on Metric Spaces: A Short Course}.
{(2022)},
{Springer}.
  
  \bibitem{dacorogna2007direct}
  {B. Dacorogna}.
{Direct methods in the calculus of variations},
{78},
{(2007)},
{Springer Science \& Business Media}.

  \bibitem{dembozeitouni}
  A. Dembo, O. Zeitouni.
  Large deviations techniques and applications,
  Applications of Mathematics, 38, 2011. 
  
  
\bibitem{ecotiere}
C. Ecoti{\`e}re, S.  Billiard, J.B. Andr{\'e}, P. Collet,  R. Ferri{\`e}re, S. M\'el\'eard.
Human-environment feedback and the consistency of proenvironmental behavior,
  Plos Computational biology (2023),
  https://doi.org/10.1371/journal.pcbi.1011429.
  
  
  \bibitem{ethier2009markov}
  {S.N. Ethier  and T.G. Kurtz}.
{\it Markov processes: characterization and convergence},
{(1986)},
{John Wiley \& Sons}

  
  
    \bibitem{freidlin1998random}
    {M.I. Freidlin, A.D. Wentzell}.
{\it Random perturbations of dynamical systems},
Second Edition,
{(1998)},
{Springer-Verlag, New York}.

  
  \bibitem{hairer2009hot}
  M. Hairer.
  How hot can a heat bath get?,
Communications in Mathematical Physics,
{292},
 no {1},
p.{131--177},
{(2009)},
{Springer}.
 
 \bibitem{hairer2009slow}
 {M. Hairer, J.C. Mattingly}.
 {Slow energy dissipation in anharmonic oscillator chains},
 {Communications on Pure and Applied Mathematics: A Journal Issued by the Courant Institute of Mathematical Sciences}, {62},
  no {8},
  p. {999--1032},
  {(2009)},
{Wiley Online Library}.


\bibitem{holbach2020positive}
{S. Holbach}.
{Positive Harris recurrence for degenerate diffusions with internal variables and randomly perturbed time-periodic input},
{Stochastic Processes and their Applications},
{130},
  no {11},
  p. {6965--7003},
{(2020)},
{Elsevier}.


\bibitem{hopfner2017strongly}
{R. H{\"o}pfner, E. L{\"o}cherbach, M. Thieullen}.
{Strongly degenerate time inhomogeneous SDEs: Densities and support properties. Application to Hodgkin--Huxley type systems},
{Bernoulli}, 
{(2017)}. 


\bibitem{hormander1967hypoelliptic}
{L. H{\"o}rmander}.
{{Hypoelliptic second order differential equations}},
{119},
 {Acta Mathematica},
{Institut Mittag-Leffler},
p. {147 -- 171},
{(1967)},
 {https://doi.org/10.1007/BF02392081}.
  
  \bibitem{bellet2006ergodic}
  {L. Rey-Bellet}.
{Ergodic properties of Markov processes},
{Open quantum systems II},
  p. {1--39},
{(2006)},
  {Springer}.

  
  
  \bibitem{stroock1972support}
  {D. Stroock, S.R.S. Varadhan}.
{On the support of diffusion processes with applications to},
{Proceedings of the Berkeley Symposium on Mathematical Statistics and Probability},
{1},
  p. {333},
{(1972)},
{University of California Press}.
  \end{thebibliography}
\end{document}